\documentclass[12pt]{article}

\usepackage{amsmath, epsfig, cite, setspace}
\usepackage{amssymb}
\usepackage{amsfonts}
\usepackage{latexsym}
\usepackage{amsthm}

\makeatletter
\renewcommand{\@seccntformat}[1]{{\csname the#1\endcsname}.\hspace{.5em}}
\makeatother

\newtheorem{theorem}{Theorem}[section]

\newtheorem{conjecture}[theorem]{Conjecture}
\newtheorem{lemma}[theorem]{Lemma}

\renewcommand{\thefootnote}{*}

\numberwithin{equation}{section}

\begin{document}

\begin{center}
{\large\bf  $q$-Supercongruences from squares of basic hypergeometric series}
\end{center}

\vskip 2mm \centerline{Victor J. W. Guo and Long Li\footnote{Corresponding author.} }
\begin{center}
{\footnotesize School of Mathematics and Statistics, Huaiyin Normal University, Huai'an 223300, Jiangsu,\\
 People's Republic of China\\
{\tt  jwguo@hytc.edu.cn,   lli@hytc.edu.cn } }
\end{center}


\vskip 0.7cm \noindent{\bf Abstract.} We give some new $q$-supercongruences on truncated forms of squares of basic hypergeometric series.
Most of them are modulo the cube of a cyclotomic polynomial, and two of them are modulo the fourth power of a cyclotomic polynomial.
The main ingredients of our proofs are the creative microscoping method, a lemma of El Bachraoui, and the Chinese remainder theorem
for coprime polynomials. We also propose several related conjectures for further study.

\vskip 3mm \noindent {\it Keywords}: $q$-supercongruences; cyclotomic polynomial; creative microscoping; basic hypergeometric series.

\vskip 0.2cm \noindent{\it AMS Subject Classifications}: 11A07, 11B65

\renewcommand{\thefootnote}{**}

\section{Introduction}

$q$-Supercongruences have attracted a lot of interest of authors in recent years.
For example, the first author and Zudilin \cite{GuoZu} devised a new method, called `creative microscoping', to prove $q$-supercongruences
by adding an extra parameter and deliberating asymptotic behavior of $q$-series at roots of unity. A representative $q$-supercongruence established by them
is as follows:  for any positive integer
$n$ with $\gcd(n,6)=1$,
\begin{align}
\sum_{k=0}^{(n-1)/2}[8k+1]\frac{(q;q^2)_k^2
(q;q^2)_{2k}}{(q^2;q^2)_{2k}(q^6;q^6)_k^2}q^{2k^2} \equiv
q^{(1-n)/2}[n]\biggl(\frac{-3}{n}\biggr) \pmod{[n]\Phi_n(q)^2}.
\label{q4b}
\end{align}
Here and in what follows, the $\emph{q-shifted factorail}$ is defined by
$(a;q)_0=1$ and $(a;q)_n=(1-a)(1-aq)\cdots (1-aq^{n-1})$ for $n\geqslant 1$,
the $\emph{q-integer}$ is defined as $[n]=1+q+\cdots+q^{n-1}$.
For simplicity, shall also use the compact notation: $(a_1,\ldots,a_m;q)_n=(a_1;q)_n\ldots (a_m;q)_n$.
Moreover, the $n$-th $\emph{cyclotomic polynomials}$, denoted by $\Phi_n(q)$, is defined by
$$\Phi_n(q)=\prod_{\substack{1\leqslant k\leqslant n\\ \gcd(k,n)=1}}(q-\zeta^k),$$ where $\zeta$ is an $n$-th primitive root of unity,
and $(\frac{-3}{n})$ denotes the Jacobi symbol.
For more progress on $q$-supercongruences, we refer the reader to
\cite{Mohamed,Guo-2018,Guo-para,Guo-a2,Guo-div,Guo-f,Guo-mod4,Guo-c2,Guo-diff,Guo-ijnt,GS0,GS2,GS,GW,GZ,GuoZu3,Li,LW,Liu,LH,LP,NP,NP2,Straub,WY,WY2,WY3,Wei,WLW,Zudilin}.

El Bachraoui \cite{Mohamed} employed the creative microscoping method to prove a few new $q$-supercongruences, such as
\begin{equation}
\sum_{k=0}^{n-1}\sum_{j=0}^{k}c_q(j)c_q(k-j)\equiv q[n]^2 \pmod{[n]\Phi_n(q)^2},  \label{eq:el}
\end{equation}
where $c_q(k)$ stands for the $k$-th term on the left-hand side of \eqref{q4b} and $\gcd(n,6)=1$.
Note that the left-hand side of \eqref{eq:el} may be deemed a truncated form of the square of the basic hypergeometric series
$\sum_{k=0}^{\infty}c_q(k)$.
Very recently, motivated by El Bachraoui's work,
the second author \cite{Li} established three more such $q$-supercongruences, including
\begin{equation}
\sum_{k=0}^{n-1}\sum_{j=0}^{k}c_q(j)c_q(k-j)\equiv q[n]^2 \pmod{[n]\Phi_n(q)^2},  \label{eq:li-1}
\end{equation}
where $c_q(k)=(-1)^kq^{k^2}[4k+1](q;q^2)_k^3/(q^2;q^2)_k^3$.

The aim of this paper is to give four $q$-supercongruences of this kind. Our first result can be stated as follows.

\begin{theorem}\label{thm1}
Let $n$ be a positive odd integer, and for $k\geqslant 0$,
$$
c_q(k)=[3k+1]\frac{(q;q^2)_k^3 q^{-{k+1\choose 2}}}{(q;q)_k^2(q^2;q^2)_k}.
$$
Then
\begin{equation}
\sum_{k=0}^{n-1}\sum_{j=0}^{k}c_q(j)c_q(k-j)\equiv q[n]^2\pmod{[n]\Phi_n(q)^2}. \label{eq:thm1}
\end{equation}
\end{theorem}

Note that the first author \cite{Guo-div} obtained the following $q$-supercongruence:
\begin{align*}
\sum_{k=0}^{n-1}[3k+1]\frac{(q;q^2)_k^3 q^{-{k+1\choose 2}}}{(q;q)_k^2(q^2;q^2)_k}
\equiv q^{(1-n)/2} \pmod{[n]\Phi_n(q)^2},
\end{align*}
which is a $q$-analogue of a `divergent' Ramanujan-type supercongruence of Guillera and Zudilin \cite{GZ}.

Our second result is a $q$-supercongruence similar to Theorem \ref{thm1}.
\begin{theorem}\label{thm2}
Let $n$ be a positive odd integer, and for $k\geqslant 0$,
$$
c_q(k)=(-1)^k[3k+1]\frac{(q;q^2)_k^3}{(q;q)_k^3}.
$$
Then
\begin{equation}
\sum_{k=0}^{n-1}\sum_{j=0}^{k}c_q(j)c_q(k-j)\equiv q^{(n+1)/2}[n]^2\pmod{[n]\Phi_n(q)^2}. \label{eq:thm2}
\end{equation}
\end{theorem}

Note that the first author \cite{Guo-div} also obtained the following $q$-supercongruence:
\begin{align*}
\sum_{k=0}^{n-1}(-1)^k[3k+1]\frac{(q;q^2)_k^3}{(q;q)_k^3}
\equiv (-q)^{(n-1)^2/4}[n] \pmod{[n]\Phi_n(q)^2},
\end{align*}
which is a $q$-analogue of another `divergent' Ramanujan-type supercongruence proved by Guillera and Zudilin \cite{GZ}.

Our third result generalizes \eqref{eq:li-1} to the modulus $[n]\Phi_n(q)^3$ case.
\begin{theorem}\label{thm3}
Let $n$ be a positive odd integer, and for $k\geqslant 0$,
$$
c_q(k)=(-1)^kq^{k^2}[4k+1]\frac{(q;q^2)_k^3}{(q^2;q^2)_k^3}.
$$
Then
\begin{equation}
\sum_{k=0}^{n-1}\sum_{j=0}^{k}c_q(j)c_q(k-j)\equiv q^{(n-1)^2/2}[n]^2\pmod{[n]\Phi_n(q)^3}. \label{eq:thm3}
\end{equation}
\end{theorem}

We point out that the first author \cite{Guo-2018} has proved that
$$
\sum_{k=0}^{(n-1)/2}(-1)^kq^{k^2}[4k+1]\frac{(q;q^2)_k^3}{(q^2;q^2)_k^3}
\equiv (-q)^{(n-1)^2/4}[n] \pmod{[n]\Phi_n(q)^2},
$$
which is a $q$-analogue of the (B.2) supercongruence of Van Hamme \cite{Hamme}. Moreover,
the first author and Wang \cite{GW} established the following $q$-supercongruence:
\begin{align}
\sum_{k=0}^{(n-1)/2}[4k+1]\frac{(q;q^2)_k^4}{(q^2;q^2)_k^4}
\equiv q^{(1-n)/2}[n]+\frac{(n^2-1)(1-q)^2}{24}q^{(1-n)/2}[n]^3 \pmod{[n]\Phi_n(q)^3},  \label{eq:gw}
\end{align}
which is a $q$-analogue of \cite[Theorem 1.1 with $r=1$]{Long} and is also a generalization of the (C.2)
supercongruence of Van Hamme \cite{Hamme}.

Our fourth result in this paper is related to \eqref{eq:gw} and can be stated as follows.
\begin{theorem}\label{thm4}
Let $n$ be a positive odd integer, and for $k\geqslant 0$,
$$
c_q(k)=[4k+1]\frac{(q;q^2)_k^4}{(q^2;q^2)_k^4}.
$$
Then
\begin{equation}
\sum_{k=0}^{n-1}\sum_{j=0}^{k}c_q(j)c_q(k-j) \equiv q^{1-n}[n]^2 \pmod{[n]\Phi_n(q)^3}.  \label{eq:thm4}
\end{equation}
\end{theorem}

The rest of the paper is arranged as follows. We shall prove Theorems \ref{thm1} and \ref{thm2} in
the next section. The proofs of Theorems \ref{thm3} and \ref{thm4} will be given in Sections 3 and 4, respectively.
In Section 5, we give four more such $q$-supercongruences.
Finally, in Section 6, we put forward eight open problems. Besides the creative microscoping method and a lemma of El Bachraoui \cite{Mohamed},
we shall also employ the Chinese remainder theorem for relatively prime polynomials to prove Theorems \ref{thm3} and \ref{thm4}.

\section{Proof of Theorems \ref{thm1} and \ref{thm2}}
We require the following two lemmas. Lemma \ref{Mohamed-Lemma1-2019} is easily proved and can be found in  \cite[Lemma 1]{Mohamed}.
\begin{lemma}\label{Mohamed-Lemma1-2019}
Let $d$ be a positive integer and let $\{c(k)\}_{k=0}^\infty$ be a sequence of complex numbers.
If $c(k)=0$ for $(d+1)/2\leqslant k\leqslant d-1$, then
$$
\sum_{k=0}^{d-1}\sum_{j=0}^{k}c(j)c(k-j)
=\Bigg(\sum_{k=0}^{d-1}c(k)\Bigg)^2.
$$
Furthermore, if $c(ld+k)/c(ld)=c(k)$ for all nonnegative integers $k$ and $l$ such that $0\leqslant k\leqslant d-1$, then
  $$\sum_{j=0}^{ld+k}c(j)c(ld+k-j)=\sum_{i=0}^{l}c(id)c((l-i)d)\sum_{j=0}^{k}c(j)c(k-j).$$
\end{lemma}

\begin{lemma}\label{lem2.1} Let $n$ be a positive odd integer. Then
\begin{align}
\sum_{k=0}^{(n-1)/2}[3k+1]\frac{(q^{1+n},q^{1-n},q;q^2)_k q^{-{k+1\choose 2}}}{(q^{1+n},q^{1-n};q)_k(q^2;q^2)_k}
&=q^{(1-n)/2}[n], \label{eq:a1}\\[5pt]
\sum_{k=0}^{(n-1)/2}(-1)^k[3k+1]\frac{(q^{1+n},q^{1-n},q;q^2)_k }{(q^{1+n},q^{1-n},q;q)_k }
&=(-q)^{(n-1)^2/4}[n].  \label{eq:a2}
\end{align}
\end{lemma}

\begin{proof}
Letting $b\to 0$ in \cite[Theorem 4.8]{GuoZu}, we obtain the following $q$-congruence: modulo $[n](1-aq^n)(a-q^n)$,
\begin{equation}
\sum_{k=0}^{(n-1)/2}[3k+1]\frac{(aq,q/a,q;q^2)_k q^{-{k+1\choose 2}}}{(aq,q/a;q)_k(q^2;q^2)_k}
\equiv q^{(1-n)/2}[n].  \label{eq:guo-1}
\end{equation}
Further taking $a=q^n$, we get \eqref{eq:a1}.

Letting $b\to 0$ and $c\to 0$ in \cite[Theorem 6.1]{GS2} (see also \cite[Conjecture 4.6]{GuoZu}), we get another $q$-congruence:
modulo $[n](1-aq^n)(a-q^n)$,
\begin{equation}
\sum_{k=0}^{(n-1)/2}(-1)^k[3k+1]\frac{(aq,q/a,q;q^2)_k }{(aq,q/a,q;q)_k }
\equiv (-q)^{(n-1)^2/4}[n].  \label{eq:guo-2}
\end{equation}
Putting $a=q^n$ in the above $q$-congruence, we arrive at \eqref{eq:a2}.
\end{proof}

We now give a parametric generalization of Theorem \ref{thm1},
which is necessary in our proof of Theorem \ref{thm1} and is also an example for the method of creative microscoping.

\begin{theorem}\label{thm:thm-a1}
Let $n$ be a positive odd integer, and for $k\geqslant 0$,
$$
c_q(k)=[3k+1]\frac{(aq,q/a,q;q^2)_k q^{-{k+1\choose 2}}}{(aq,q/a;q)_k(q^2;q^2)_k}.
$$
Then
\begin{equation}
\sum_{k=0}^{n-1}\sum_{j=0}^{k}c_q(j)c_q(k-j)\equiv q^{1-n}[n]^2\pmod{[n](1-aq^n)(a-q^n)}. \label{eq:thm-a1}
\end{equation}
\end{theorem}

\begin{proof}
For $n=1$, the result is clearly true. We now assume that $n$ is an odd integer greater than $1$.
Let $\zeta\neq 1$ be an $n$-th root of unity, not necessarily primitive. Namely, $\zeta$ is a primitive $d$-th root of unity with $d\mid n$ and $d>1$.
It is easy to see that $(\zeta;\zeta^2)_k=0$ for $(d+1)/2\leqslant k\leqslant d-1$ and so $c_\zeta(k)=0$ for $k$ in the same range.
By Lemma \ref{Mohamed-Lemma1-2019}, we obtain
$$
\sum_{k=0}^{d-1}\sum_{j=0}^{k}c_\zeta(j)c_\zeta(k-j)=\left(\sum_{k=0}^{d-1}c_\zeta(k)\right)^2=0,
$$
where the second equality follows from the $n=d$ case of \eqref{eq:guo-1}.
It is not difficult to check that $c_\zeta(ld+k)/c_\zeta(ld)=c_\zeta(k)$ for all integers $k$ and $l$ satisfying $0\leqslant k\leqslant d-1$.
By Lemma \ref{Mohamed-Lemma1-2019} again, we get
\begin{align*}
\sum_{m=0}^{n-1}\sum_{j=0}^{m}c_\zeta(j)c_\zeta(m-j)
&=\sum_{l=0}^{n/d-1}\sum_{k=0}^{d-1}\sum_{j=0}^{ld+k}c_\zeta(j)c_\zeta(ld+k-j)\\
&=\sum_{l=0}^{n/d-1}\sum_{k=0}^{d-1}\sum_{i=0}^{l}\left(c_\zeta(id)c_\zeta((l-i)d)\right)\sum_{j=0}^{k}c_\zeta(j)c_\zeta(k-j)\\
&=\Bigg(\sum_{l=0}^{n/d-1}\sum_{i=0}^{l}c_\zeta(id)c_\zeta((l-i)d)\Bigg)\sum_{k=0}^{d-1}\sum_{j=0}^{k}c_\zeta(j)c_\zeta(k-j)\\
&=0.
\end{align*}
Since the above equality is true for any $n$-th root of unit $\zeta\neq 1$, we conclude that
\begin{equation}\label{eq2.5}
\sum_{k=0}^{n-1}\sum_{j=0}^{k}c_q(j)c_q(k-j)\equiv 0\equiv q^{1-n}[n]^2 \pmod{[n]}.
\end{equation}
Namely, the $q$-congruence \eqref{eq:thm-a1} holds modulo $[n]$.

Moreover, for $a=q^n$ or $a=q^{-n}$, in view of \eqref{eq:a1}, we have
$
\sum_{k=0}^{n-1}c_q(k)=q^{(1-n)/2}[n],
$
and $c_q(k)=0$ for $k$ in the range $(n+1)/2\leqslant k\leqslant n-1$.
Thus, by Lemma \ref{Mohamed-Lemma1-2019}, for $a=q^n$ or $a=q^{-n}$ we get
\begin{equation*}
\sum_{k=0}^{n-1}\sum_{j=0}^{k}c_q(j)c_q(k-j)=\Bigg(\sum_{k=0}^{n-1}c_q(k)\Bigg)^2=q^{1-n}[n]^2.
\end{equation*}
This means that the $q$-congruence \eqref{eq:thm-a1} holds modulo $1-aq^n$ and $a-q^n$.
The proof then follows from the fact that the polynomials $[n]$, $1-aq^n$ and $a-q^n$ are pairwise relatively prime.
\end{proof}

\begin{proof}[Proof of Theorem {\rm\ref{thm1}}]
Setting $a=1$ in \eqref{eq:thm-a1}, and observing that $1-q^n$ contains the factor $\Phi_n(q)$, we
are led to the $q$-congruence \eqref{eq:thm1}  modulo $\Phi_n(q)^3$.
Furthermore, our proof of \eqref{eq2.5} is also true for $a=1$. That is,
the $q$-congruence \eqref{eq:thm1} holds modulo $[n]$. Noticing that the least common multiple of
$\Phi_n(q)^3$ and $[n]$ is just $[n]\Phi_n(q)^2$, we complete the proof of the theorem.
\end{proof}

Similarly, we have a parametric generalization of Theorem \ref{thm2}.
\begin{theorem}
Let $n$ be a positive odd integer, and for $k\geqslant 0$,
$$
c_q(k)=(-1)^k[3k+1]\frac{(aq,q/a,q;q^2)_k }{(aq,q/a,q;q)_k }.
$$
Then
\begin{equation}
\sum_{k=0}^{n-1}\sum_{j=0}^{k}c_q(j)c_q(k-j)\equiv q^{(n-1)^2/2}[n]^2\pmod{[n](1-aq^n)(a-q^n)}. \label{eq:thm-a2}
\end{equation}
\end{theorem}

\begin{proof}The proof is very similar to that of Theorem \ref{thm:thm-a1}.
We only consider the $n>1$ case.
Let $\zeta$ be a primitive $d$-th root of unity with $d\mid n$ and $d>1$.
Then $c_\zeta(k)=0$ for $(d+1)/2\leqslant k\leqslant d-1$.
By Lemma \ref{Mohamed-Lemma1-2019}, we get
$$
\sum_{k=0}^{d-1}\sum_{j=0}^{k}c_\zeta(j)c_\zeta(k-j)=\left(\sum_{k=0}^{d-1}c_\zeta(k)\right)^2=0,
$$
where we have used the $n=d$ case of \eqref{eq:guo-2}.
Likewise, $c_\zeta(ld+k)/c_\zeta(ld)=c_\zeta(k)$ is also true for $0\leqslant k\leqslant d-1$.
By Lemma \ref{Mohamed-Lemma1-2019} again, we get
\begin{align*}
\sum_{m=0}^{n-1}\sum_{j=0}^{m}c_\zeta(j)c_\zeta(m-j)
=0,
\end{align*}
from which we conclude that
\begin{equation}\label{eq2.9}
\sum_{k=0}^{n-1}\sum_{j=0}^{k}c_q(j)c_q(k-j)\equiv 0\equiv q^{(1-n)^2/2}[n]^2 \pmod{[n]}.
\end{equation}

For $a=q^n$ or $a=q^{-n}$, in view of \eqref{eq:a2}, we have
$
\sum_{k=0}^{n-1}c_q(k)=(-q)^{(n-1)^2/4}[n],
$
and $c_q(k)=0$ for $(n+1)/2\leqslant k\leqslant n-1$.
Hence, by Lemma \ref{Mohamed-Lemma1-2019}, for $a=q^n$ or $q^{-n}$ we get
\begin{equation*}
\sum_{k=0}^{n-1}\sum_{j=0}^{k}c_q(j)c_q(k-j)=q^{(n-1)^2/2}[n]^2.
\end{equation*}
This proves that the $q$-congruence \eqref{eq:thm-a2} holds modulo $1-aq^n$ and $a-q^n$.
\end{proof}

\begin{proof}[Proof of Theorem {\rm\ref{thm3}}]
Letting $a=1$ in \eqref{eq:thm-a2} and observing that $q^{(n-1)^2/2}\equiv q^{(n+1)/2}\pmod{\Phi_n(q)}$,
we obtain the desired $q$-congruence \eqref{eq:thm2}.
\end{proof}

\section{Proof of Theorem \ref{thm3}}
We first give two assistant results. Both of them can be deduced form Watson's $_8\phi_7$
transformation formula \cite[Appendix (III. 18)]{GR}. For the concrete proofs, see \cite[Section~5]{Guo-mod4}.
\begin{lemma}Let $n$ be a positive odd integer. Then, modulo $[n](1-aq^n)(a-q^n)$,
\begin{align}
&\sum_{k=0}^{(n-1)/2}(-1)^k [4k+1]\frac{(aq,q/a,q/b;q^2)_k}{(aq^2,q^2/a,bq^2;q^2)_k}b^k q^{k^2} \notag\\[5pt]
&\quad\equiv (-q)^{(n-1)^2/4}[n]\sum_{k=0}^{(n-1)/2}\frac{(1-b)(aq,q/a;q^2)_k}{(1-bq^{2k})(q,q^2;q^2)_k}q^k.
\label{eq:false-1}
\end{align}
\end{lemma}

\begin{lemma}Let $n$ be a positive odd integer. Then, modulo $b-q^n$,
\begin{align}
\sum_{k=0}^{(n-1)/2}(-1)^k [4k+1]\frac{(aq,q/a,q/b;q^2)_k}{(aq^2,q^2/a,bq^2;q^2)_k}b^k q^{k^2}
\equiv [n]\sum_{k=0}^{(n-1)/2}\frac{(q,q/b;q^2)_k b^k}{(aq^2,q^2/a;q^2)_k}.  \label{eq:false-2}
\end{align}
\end{lemma}

We also need the following lemma.
\begin{lemma}\label{lemma:zero}
Let $n$ be a positive odd integer. Let $a_0,a_1,\ldots,a_{n-1}$ be a sequence of numbers satisfying
$a_k=-a_{(n-1)/2-k}$ for $0\leqslant k\leqslant (n-1)/2$ and $a_{k}=-a_{(3n-1)/2-k}$ for $(n+1)/2\leqslant k\leqslant n-1$.
Then
\begin{equation}
\sum_{k=0}^{n-1}\sum_{i=0}^{k}a_i a_{k-i}=0. \label{eq:zero}
\end{equation}
\end{lemma}
\begin{proof}We consider two cases. If $n\equiv 1\pmod{4}$, then $a_{(n-1)/4}=0$, and for nonnegative indices $i,j\neq (n-1)/4$ and $i+j\leqslant n-1$, we have
\begin{equation}
a_ia_{j}
=\begin{cases}
-a_i a_{(n-1)/2-j} &\text{if $i,j\leqslant (n-1)/2$},\\[5pt]
-a_i a_{(3n-1)/2-j} &\text{if $i\leqslant (n-1)/2$ and $j\geqslant (n+1)/2$},\\[5pt]
-a_{(3n-1)/2-i} a_{j} &\text{if $i\geqslant (n+1)/2$ and $j\leqslant (n-1)/2$},\\[5pt]
\end{cases} \label{eq:cases}
\end{equation}
This means that all the nonzero terms on the left-hand side of \eqref{eq:zero} can be paired so that the two terms in each pair only differs  by a sign.
Therefore, the identity \eqref{eq:zero} holds.

If $n\equiv 3\pmod{4}$, then $a_{(3n-1)/4}=0$, and for nonnegative indices $i,j\neq (3n-1)/4$ and $i+j\leqslant n-1$, the equality
\eqref{eq:cases} still  holds, thus establishing \eqref{eq:zero}.
\end{proof}

We are now able to establish the following $q$-congruence with an additional parameter $a$. Note that this $q$-congruence
modulo $[n](1-aq^n)(a-q^n)$ was already given by the second author \cite{Li}.
\begin{theorem}\label{thm-a3}
Let $n$ be a positive odd integer, and for $k\geqslant 0$,
$$
c_q(k)=(-1)^kq^{k^2}[4k+1]\frac{(aq,q/a,q;q^2)_k}{(aq^2,q^2/a,q^2;q^2)_k}.
$$
Then
\begin{equation}
\sum_{k=0}^{n-1}\sum_{j=0}^{k}c_q(j)c_q(k-j)\equiv q^{(n-1)^2/2}[n]^2\pmod{[n]\Phi_n(q)(1-aq^n)(a-q^n)}. \label{eq:thm-a3}
\end{equation}
\end{theorem}

\begin{proof}
Suppose that $n>1$ is odd. Let
$$
z_q(k)=(-1)^k [4k+1]\frac{(aq,q/a,q/b;q^2)_k}{(aq^2,q^2/a,bq^2;q^2)_k}b^k q^{k^2} .
$$
By \cite[Lemma 3.1]{GS}, for $0\leqslant k\leqslant (n-1)/2$, we have
\begin{equation}
\frac{(aq;q^2)_{(n-1)/2-k}}{(q^2/a;q^2)_{(n-1)/2-k}}
\equiv (-a)^{(n-1)/2-2k}\frac{(aq;q^2)_k}{(q^2/a;q^2)_k} q^{(n-1)^2/4+k}
\pmod{\Phi_n(q)}.  \label{eq:gs}
\end{equation}
From the above $q$-congruence, we can easily check that, for $0\leqslant k\leqslant (n-1)/2$,
\begin{equation}
z_q(k)\equiv -z_q((n-1)/2-k) \pmod{\Phi_n(q)}.  \label{eq:zqk-1}
\end{equation}
Similarly, for $(n+1)/2\leqslant k\leqslant n-1$,
\begin{equation}
z_q(k)\equiv -z_q((3n-1)/2-k) \pmod{\Phi_n(q)}.  \label{eq:zqk-2}
\end{equation}
By Lemma \ref{lemma:zero}, we get
\begin{equation}
\sum_{k=0}^{n-1}\sum_{j=0}^{k}z_q(j)z_q(k-j)\equiv 0\pmod{\Phi_n(q)}. \label{eq:thm-a300}
\end{equation}

Let $\zeta$ be a primitive $d$-th root of unity with $d\mid n$ and $d>1$. Then the $n=d$ case of \eqref{eq:thm-a300}
implies that
$$
\sum_{k=0}^{d-1}\sum_{j=0}^{k}z_\zeta(j)z_\zeta(k-j)=0,
$$
Moreover, the equality $z_\zeta(ld+k)/z_\zeta(ld)=c_\zeta(k)$ holds for $0\leqslant k\leqslant d-1$.
By Lemma \ref{Mohamed-Lemma1-2019}, similarly to the proof of \eqref{eq2.5}, we can show that
the $q$-congruence \eqref{eq:thm-a300} is true modulo $[n]$.

For $a=q^n$ or $a=q^{-n}$, by \eqref{eq:false-1}, we have
$$
\sum_{k=0}^{n-1}z_q(k)=(-q)^{(n-1)^2/4}[n]\sum_{k=0}^{(n-1)/2}\frac{(1-b)(q^{1+n},q^{1-n};q^2)_k}{(1-bq^{2k})(q,q^2;q^2)_k}q^k,
$$
and $z_q(k)=0$ for $(n+1)/2\leqslant k\leqslant n-1$. By Lemma \ref{Mohamed-Lemma1-2019}, we get
\begin{equation*}
\sum_{k=0}^{n-1}\sum_{j=0}^{k}z_q(j)z_q(k-j)
=q^{(n-1)^2/2}[n]^2\left(\sum_{k=0}^{(n-1)/2}\frac{(1-b)(aq^{1+n},q^{1-n};q^2)_k}{(1-bq^{2k})(q,q^2;q^2)_k}q^k \right)^2.
\end{equation*}
Namely, for the indeterminates $a$ and $b$, we obtain the following $q$-congruence: modulo $(1-aq^n)(a-q^n)$,
\begin{equation}
\sum_{k=0}^{n-1}\sum_{j=0}^{k}z_q(j)z_q(k-j)
\equiv q^{(n-1)^2/2}[n]^2\left(\sum_{k=0}^{(n-1)/2}\frac{(1-b)(aq,q/a;q^2)_k}{(1-bq^{2k})(q,q^2;q^2)_k}q^k \right)^2.
\label{eq:chinese-a}
\end{equation}
Similarly, from \eqref{eq:false-2} we deduce the following $q$-congruence: modulo $b-q^n$,
\begin{equation}
\sum_{k=0}^{n-1}\sum_{j=0}^{k}z_q(j)z_q(k-j)
\equiv [n]^2\left(\sum_{k=0}^{(n-1)/2}\frac{(q,q/b;q^2)_k b^k}{(aq^2,q^2/a;q^2)_k}\right)^2.
\label{eq:chinese-b}
\end{equation}

It is clear that $(1-aq^n)(a-q^n)$, and $b-q^n$ are relatively prime the polynomials.
In light of the Chinese reminder theorem for polynomials, we can determine the remainder of the left-hand side of \eqref{eq:thm-a300} modulo $[n](1-aq^n)(a-q^n)(b-q^n)$
from \eqref{eq:chinese-a} and \eqref{eq:chinese-b}.
In fact, using following two $q$-congruences:
\begin{align}
\frac{(b-q^n)(ab-1-a^2+aq^n)}{(a-b)(1-ab)}&\equiv 1\pmod{(1-aq^n)(a-q^n)}, \label{eq:ab-1} \\[5pt]
\frac{(1-aq^n)(a-q^n)}{(a-b)(1-ab)}&\equiv 1\pmod{b-q^n}.  \label{eq:ab-2}
\end{align}
we conclude that
\begin{align}
&\sum_{k=0}^{n-1}\sum_{j=0}^{k}z_q(j)z_q(k-j)\notag\\[5pt]
&\quad\equiv\frac{(b-q^n)(ab-1-a^2+aq^n)}{(a-b)(1-ab)}q^{(n-1)^2/2}[n]^2\left(\sum_{k=0}^{(n-1)/2}\frac{(1-b)(aq,q/a;q^2)_k}{(1-bq^{2k})(q,q^2;q^2)_k}q^k \right)^2 \notag\\[5pt]
&\quad\quad+\frac{(1-aq^n)(a-q^n)}{(a-b)(1-ab)}[n]^2\left(\sum_{k=0}^{(n-1)/2}\frac{(q,q/b;q^2)_k b^k}{(aq^2,q^2/a;q^2)_k}\right)^2 \label{eq:abqn}
\end{align}
modulo $[n](1-aq^n)(a-q^n)(b-q^n)$.

We now take $b=1$ in \eqref{eq:abqn}. In this case, the polynomial $b-q^n=1-q^n$ contains the factor $\Phi_n(q)$.
Meanwhile, the second part on the right-hand of \eqref{eq:abqn} modulo $[n]\Phi_n(q)(1-aq^n)(a-q^n)$ vanishes.
Therefore, the $q$-congruence \eqref{eq:abqn} reduces to the following one: modulo $[n]\Phi_n(q)(1-aq^n)(a-q^n)$,
\begin{align*}
\sum_{k=0}^{n-1}\sum_{j=0}^{k}c_q(j)c_q(k-j)
\equiv\frac{(1-q^n)(1+a^2-a-aq^n)}{(1-a)^2}q^{(n-1)^2/2}[n]^2,
\end{align*}
which is clearly equivalent to \eqref{eq:thm-a3}, since
\begin{equation}
(1-q^n)(1+a^2-a-aq^n)=(1-a)^2+(1-aq^n)(a-q^n).  \label{eq:relation}
\end{equation}
This completes the proof.
\end{proof}

\begin{proof}[Proof of Theorem {\rm\ref{thm3}}]
Letting $a=1$ in \eqref{eq:thm-a3}, we are led to \eqref{eq:thm3}.
\end{proof}

\section{Proof of Theorem \ref{thm4}}
The proof is analogous to that of Theorem \ref{thm3}. This time we need the following two auxiallary results, which can be derived form Jackson's $_6\phi_5$
summation formula \cite[ Appendix (II.21)]{GR}. For the detailed proofs of them, see \cite[Theorem 4.2]{GuoZu} and \cite[Lemma 2.3]{Guo-mod4}.

\begin{lemma}
\label{th:4.2}
Let $n$ be a positive odd integer. Then, modulo $[n](1-aq^n)(a-q^n)$,
\begin{equation}
\sum_{k=0}^{(n-1)/2}[4k+1]\frac{(aq,q/a,q/b,q;q^2)_k}
{(aq^2,q^2/a,bq^2,q^2;q^2)_k}b^k
\equiv\frac{(b/q)^{(n-1)/2} (q^2/b;q^2)_{(n-1)/2}}{(bq^2;q^2)_{(n-1)/2}}[n].
\label{eq:q-Long-gen}
\end{equation}
\end{lemma}

\begin{lemma}\label{lem:2}
Let $n$ be a positive odd integer. Then, modulo $b-q^n$,
\begin{equation}
\sum_{k=0}^{(n-1)/2}[4k+1]\frac{(aq,q/a,q/b,q;q^2)_k}
{(aq^2,q^2/a,bq^2,q^2;q^2)_k}b^k
\equiv\frac{(q;q^2)_{(n-1)/2}^2 [n]}{(aq^2,q^2/a;q^2)_{(n-1)/2}}.
\label{eq:q-Long-gen-2}
\end{equation}
\end{lemma}

Likewise, we first give the following parametric generalization of Theorem \ref{thm4}.
\begin{theorem}\label{thm-a4}
Let $n$ be a positive odd integer, and for $k\geqslant 0$,
$$
c_q(k)=[4k+1]\frac{(aq,q/a;q^2)_k(q;q^2)_k^2}{(aq^2,q^2/a;q^2)_k(q^2;q^2)_k^2}.
$$
Then
\begin{equation}
\sum_{k=0}^{n-1}\sum_{j=0}^{k}c_q(j)c_q(k-j)\equiv q^{1-n}[n]^2\pmod{[n]\Phi_n(q)(1-aq^n)(a-q^n)}. \label{eq:thm-a4}
\end{equation}
\end{theorem}

\begin{proof}
Suppose that $n>1$ is odd. Let
$$
z_q(k)=[4k+1]\frac{(aq,q/a,q/b,q;q^2)_k}
{(aq^2,q^2/a,bq^2,q^2;q^2)_k}b^k.
$$
By \eqref{eq:gs}, we have
$
z_q(k)\equiv -z_q((n-1)/2-k) \pmod{\Phi_n(q)}
$
for $0\leqslant k\leqslant (n-1)/2$.
Similarly, we also have
$
z_q(k)\equiv -z_q((3n-1)/2-k) \pmod{\Phi_n(q)}
$
for $(n+1)/2\leqslant k\leqslant n-1$.
In view of Lemma \ref{lemma:zero}, we get
\begin{equation}
\sum_{k=0}^{n-1}\sum_{j=0}^{k}z_q(j)z_q(k-j)\equiv 0\pmod{\Phi_n(q)}. \label{eq:thm-a400}
\end{equation}
Like before, we can further show that above $q$-congruence is true modulo $[n]$.

For $a=q^n$ or $a=q^{-n}$, by \eqref{eq:q-Long-gen}, we have
$$
\sum_{k=0}^{n-1}z_q(k)=\frac{(b/q)^{(n-1)/2} (q^2/b;q^2)_{(n-1)/2}}{(bq^2;q^2)_{(n-1)/2}}[n],
$$
and $z_q(k)=0$ for $(n+1)/2\leqslant k\leqslant n-1$. By Lemma \ref{Mohamed-Lemma1-2019}, we get
\begin{equation*}
\sum_{k=0}^{n-1}\sum_{j=0}^{k}z_q(j)z_q(k-j)
=\frac{(b/q)^{n-1} (q^2/b;q^2)_{(n-1)/2}^2}{(bq^2;q^2)_{(n-1)/2}^2}[n]^2.
\end{equation*}
Namely, for the indeterminates $a$ and $b$, we obtain
\begin{equation}
\sum_{k=0}^{n-1}\sum_{j=0}^{k}z_q(j)z_q(k-j)
\equiv \frac{(b/q)^{n-1} (q^2/b;q^2)_{(n-1)/2}^2}{(bq^2;q^2)_{(n-1)/2}^2}[n]^2 \pmod{(1-aq^n)(a-q^n)}.
\label{eq:chinese-a2}
\end{equation}
Similarly, applying \eqref{eq:q-Long-gen-2} we get the $q$-congruence: modulo $b-q^n$,
\begin{equation}
\sum_{k=0}^{n-1}\sum_{j=0}^{k}z_q(j)z_q(k-j)
\equiv \frac{(q;q^2)_{(n-1)/2}^4 [n]^2}{(aq^2,q^2/a;q^2)_{(n-1)/2}^2}.
\label{eq:chinese-b2}
\end{equation}

It follows from \eqref{eq:ab-1}, \eqref{eq:ab-2}, \eqref{eq:chinese-a2} and \eqref{eq:chinese-b2} that,  modulo $[n](1-aq^n)(a-q^n)(b-q^n)$,
\begin{align}
&\sum_{k=0}^{n-1}\sum_{j=0}^{k}z_q(j)z_q(k-j)\notag\\[5pt]
&\quad\equiv\frac{(b-q^n)(ab-1-a^2+aq^n)}{(a-b)(1-ab)}\frac{(b/q)^{n-1} (q^2/b;q^2)_{(n-1)/2}^2}{(bq^2;q^2)_{(n-1)/2}^2}[n]^2  \notag\\[5pt]
&\quad\quad+\frac{(1-aq^n)(a-q^n)}{(a-b)(1-ab)}\frac{(q;q^2)_{(n-1)/2}^4 [n]^2}{(aq^2,q^2/a;q^2)_{(n-1)/2}^2}. \label{eq:abqn-2nd}
\end{align}
Putting $b=1$ in \eqref{eq:abqn-2nd} and applying \eqref{eq:relation}, we arrive at \eqref{eq:thm-a4}.
\end{proof}

\begin{proof}[Proof of Theorem {\rm\ref{thm4}}]
Letting $a=1$ in \eqref{eq:thm-a4}, we immediately get \eqref{eq:thm3}.
\end{proof}

\section{More such $q$-supercongruences}
The first author and Zeng \cite{GZ} gave the following $q$-supercongruecnes: for any odd prime $p$,
$$
\sum_{k=0}^{p-1}\frac{2(q;q^2)_k^2 q^{2k}}{(q^2;q^2)_k^2(1+q^{2k})}
\equiv (-1)^{(p-1)/2} \pmod{[p]^2},
$$
which is a $q$-analogue of a classical supercongruence conjectured by Rodriguez-Villegas \cite[(36)]{RV} and first confirmed by
Mortenson \cite{Mortenson}.  Here we give the corresponding $q$-supercogruence on double sums.

\begin{theorem}\label{thm5}
Let $n$ be a positive odd integer, and for $k\geqslant 0$,
$$
c_q(k)=\frac{2(q;q^2)_k^2 q^{2k}}{(q^2;q^2)_k^2(1+q^{2k})}.
$$
Then
\begin{equation}
\sum_{k=0}^{n-1}\sum_{j=0}^{k}c_q(j)c_q(k-j)\equiv 1\pmod{\Phi_n(q)^2}. \label{eq:thm5-1}
\end{equation}
\end{theorem}

\begin{proof}
By \cite[Corollary 1.4]{Guo-para}, we have
\begin{equation}
\sum_{k=0}^{n-1}\frac{2(aq,q/a;q^2)_k
q^{2k}}{(q^2;q^2)_{k}^2 (1+q^{2k})} \equiv (-1)^{(n-1)/2}\pmod{(1-aq^n)(a-q^n)}. \label{eq:more-1}
\end{equation}
Let $z_q(k)$ denote the $k$-th term on the left-hand side of \eqref{eq:more-1}.
By Lemma \ref{Mohamed-Lemma1-2019}, we can easily prove that
\begin{equation}
\sum_{k=0}^{n-1}\sum_{j=0}^{k}z_q(j)z_q(k-j)\equiv 1\pmod{(1-aq^n)(a-q^n)}. \label{eq:thm5-2}
\end{equation}
Taking $a=1$ in the above congruence, we obtain the desired $q$-congruence \eqref{eq:thm5-1}.
\end{proof}

The first author and Zudilin \cite[Theorem~2]{GuoZu2} gave the following $q$-supercongruence:
modulo $\Phi_n(q)^2$,
\begin{align*}
\sum_{k=0}^{(n-1)/2}\frac{(q;q^2)_k^2(q^2;q^4)_k}{(q^2;q^2)_k^2(q^4;q^4)_k}q^{2k}
\equiv\begin{cases}
\dfrac{(q^2;q^4)_{(n-1)/4}^2}{(q^4;q^4)_{(n-1)/4}^2}q^{(n-1)/2} &\text{if}\; n\equiv1\pmod4, \\[5pt]
0 &\text{if}\; n\equiv3\pmod4,
\end{cases}
\end{align*}
which is a $q$-analogue of the (H.2) supercongruence of Van Hamme \cite{Hamme}.
The first author \cite{Guo-ijnt} further showed that, for $n\equiv 3\pmod{4}$,
\begin{align}
\sum_{k=0}^{(n-1)/2}\frac{(q;q^2)_k^2(q^2;q^4)_k}{(q^2;q^2)_k^2(q^4;q^4)_k}q^{2k}
\equiv [n]\frac{(q^3;q^4)_{(n-1)/2}}{(q^5;q^4)_{(n-1)/2}}
\pmod{\Phi_n(q)^3},  \label{eq:h-1}
\end{align}
which is also true modulo $\Phi_n(q)^2$ for $n\equiv 1\pmod{4}$. Recently,
Wei \cite{Wei} proved that, for $n\equiv 1\pmod{4}$, modulo $\Phi_n(q)^3$,
\begin{align}
\sum_{k=0}^{(n-1)/2}\frac{(q;q^2)_k^2(q^2;q^4)_k}{(q^2;q^2)_k^2(q^4;q^4)_k}q^{2k}
\equiv  \dfrac{(q^2;q^4)_{(n-1)/4}^2}{(q^4;q^4)_{(n-1)/4}^2}q^{(n-1)/2}
\left(1+2[n]^2\sum_{k=0}^{(n-1)/4}\frac{q^{4k-2}}{[4k-2]^2}\right).  \label{eq:h-2}
\end{align}
It should be mentioned that \eqref{eq:h-1} and \eqref{eq:h-2} may be considered as a $q$-analogue of \cite[Theorem 3]{LR}.

Here we give a $q$-supercongruence on double sums related to \eqref{eq:h-1} and \eqref{eq:h-2}.

\begin{theorem}\label{thm6}
Let $n$ be a positive odd integer, and for $k\geqslant 0$,
$$
c_q(k)=\frac{(q;q^2)_k^2(q^2;q^4)_k}{(q^2;q^2)_k^2(q^4;q^4)_k}q^{2k}.
$$
Then, modulo $\Phi_n(q)^3$,
\begin{align}
&\sum_{k=0}^{n-1}\sum_{j=0}^{k}c_q(j)c_q(k-j)  \notag\\[5pt]
&\quad\equiv \begin{cases}
\dfrac{(q^2;q^4)_{(n-1)/4}^4}{(q^4;q^4)_{(n-1)/4}^4}q^{n-1}
\displaystyle\left(1+4[n]^2\sum_{k=0}^{(n-1)/4}\frac{q^{4k-2}}{[4k-2]^2}\right)  &\text{if}\; n\equiv1\pmod4, \\[10pt]
0&\text{if}\; n\equiv3\pmod4.
\end{cases}\label{eq:thm6-1}
\end{align}
\end{theorem}

\begin{proof}We first consider the $n\equiv 1\pmod{4}$ case. For $k\geqslant 0$, let
$$
z_q(k)=\frac{(aq,q/a,q/b,-q/b;q^2)_k}{(q^2,q^2,-q^2,q^2/b^2;q^2)_k}q^{2k}.
$$
Using the following two congruences  (see \cite{Wei})
\begin{align*}
\sum_{k=0}^{(n-1)/2}z_q(k)
&\equiv  \frac{(q^2,b^2q^2;q^4)_{(n-1)/4}}{(q^4,q^4/b^2;q^4)_{(n-1)/4}}\left(\frac{q}{b}\right)^{(n-1)/2}\pmod{(1-aq^n)(a-q^n)}, \\[5pt]
\sum_{k=0}^{(n-1)/2}z_q(k)
&\equiv\frac{(aq^3,q^3/a;q^4)_{(n-1)/2}}{(q^2;q^2)_{n-1}} \notag\\[5pt]
&\equiv \frac{(abq^2,bq^2/a,aq^2/b,q^2/ab;q^4)_{(n-1)/4}}{(q^2,q^4,q^2/b^2,q^4/b^2;q^4)_{(n-1)/4}}\left(\frac{q}{b}\right)^{(n-1)/2}\pmod{b-q^n},
\end{align*}
which may be deduced from Andrews' $q$-analogue of the Whipple formula \cite{Andrews} and Jain's $q$-analogue of the Whipple formula \cite{Jain}, respectively,
in view of Lemma \ref{Mohamed-Lemma1-2019}, we can prove that
\begin{align*}
\sum_{k=0}^{n-1}\sum_{j=0}^{k}z_q(j)z_q(k-j)
&\equiv  \frac{(q^2,b^2q^2;q^4)_{(n-1)/4}^2 }{(q^4,q^4/b^2;q^4)_{(n-1)/4}^2}\left(\frac{q}{b}\right)^{n-1} \pmod{(1-aq^n)(a-q^n)}, \\[5pt]
\sum_{k=0}^{n-1}\sum_{j=0}^{k}z_q(j)z_q(k-j)
&\equiv\frac{(abq^2,bq^2/a,aq^2/b,q^2/ab;q^4)_{(n-1)/4}^2}{(q^2,q^4,q^2/b^2,q^4/b^2;q^4)_{(n-1)/4}^2}\left(\frac{q}{b}\right)^{n-1} \pmod{b-q^n}.
\end{align*}
By \eqref{eq:ab-1} and \eqref{eq:ab-2}, we conclude that, modulo $(1-aq^n)(a-q^n)(b-q^n)$,
\begin{align}
&\sum_{k=0}^{n-1}\sum_{j=0}^{k}z_q(j)z_q(k-j)\notag\\[5pt]
&\quad\equiv\frac{(b-q^n)(ab-1-a^2+aq^n)}{(a-b)(1-ab)}\frac{(q^2,b^2q^2;q^4)_{(n-1)/4}^2 }{(q^4,q^4/b^2;q^4)_{(n-1)/4}^2}\left(\frac{q}{b}\right)^{n-1} \notag\\[5pt]
&\quad\quad+\frac{(1-aq^n)(a-q^n)}{(a-b)(1-ab)}\frac{(abq^2,bq^2/a,aq^2/b,q^2/ab;q^4)_{(n-1)/4}^2}{(q^2,q^4,q^2/b^2,q^4/b^2;q^4)_{(n-1)/4}^2}\left(\frac{q}{b}\right)^{n-1} \label{eq:abqn-h}
\end{align}

Putting $b=1$ in \eqref{eq:abqn-h}, we are led to the congruence: modulo $\Phi_n(q)(1-aq^n)(a-q^n)$,
\begin{align}
\sum_{k=0}^{n-1}\sum_{j=0}^{k}y_q(j)y_q(k-j)
&\equiv
\frac{(q^2;q^4)_{(n-1)/4}^4}{(q^4;q^4)_{(n-1)/4}^4}q^{n-1}+\frac{(1-aq^n)(a-q^n)}{(1-a)^2}q^{n-1}
\notag\\[5pt]
&\quad\times\left(\frac{(q^2;q^4)_{(n-1)/4}^4}{(q^4;q^4)_{(n-1)/4}^4}-\frac{(aq^2,q^2/a;q^4)_{(n-1)/4}^4}{(q^2,q^4;q^4)_{(n-1)/4}^4}\right),  \label{eq:abqn-h-2}
\end{align}
where
$$
y_q(k)=\frac{(aq,q/a,q,-q;q^2)_k}{(q^2,q^2,-q^2,q^2;q^2)_k}q^{2k}.
$$
By L'H\^ospital's rule, we have
\begin{align*}
&\lim_{a\to1}\frac{(1-aq^n)(a-q^n)}{(1-a)^2}\left(\frac{(q^2;q^4)_{(n-1)/4}^4}{(q^4;q^4)_{(n-1)/4}^4}
-\frac{(aq^2,q^2/a;q^4)_{(n-1)/4}^4}{(q^2,q^4;q^4)_{(n-1)/4}^4}\right)\\[5pt]
&=4[n]^2\frac{(q^2;q^4)_{(n-1)/4}^4}{(q^4;q^4)_{(n-1)/4}^4}\sum_{k=1}^{(n-1)/4}\frac{q^{4k-2}}{[4k-2]^2}.
\end{align*}
Thus, letting $a\to1$ in \eqref{eq:abqn-h-2}, we arrive at the first case of \eqref{eq:thm6-1}.

We now consider the $n\equiv 3\pmod{4}$ case. We need to introduce another parametric generalization. For $k\geqslant 0$, let
$$
x_q(k)=\frac{(aq,q/a;q^2)_k (q^2;q^4)_k}{(aq^2,q^2/a;q^2)_k (q^4;q^4)_k}q^{2k}
$$
The first author \cite{Guo-ijnt} gave the following parametric generalization of \eqref{eq:h-1}: modulo $\Phi_n(q)(1-aq^n)(a-q^n)$,
\begin{align*}
\sum_{k=0}^{(n-1)/2}x_q(k)
\equiv [n]\frac{(q^3;q^4)_{(n-1)/2}}{(q^5;q^4)_{(n-1)/2}}.
\end{align*}
Like before, using this congruence and Lemma \ref{Mohamed-Lemma1-2019}, we can prove that
\begin{align*}
\sum_{k=0}^{n-1}\sum_{j=0}^{k}x_q(j)x_q(k-j)
\equiv [n]^2\frac{(q^3;q^4)_{(n-1)/2}^2}{(q^5;q^4)_{(n-1)/2}^2} \pmod{(1-aq^n)(a-q^n)}.
\end{align*}
Moreover, from \eqref{eq:gs} we can easily verify that \eqref{eq:zqk-1} and \eqref{eq:zqk-2}
also hold in this case.
In light of Lemma \ref{lemma:zero}, we conclude that
\begin{equation*}
\sum_{k=0}^{n-1}\sum_{j=0}^{k}z_q(j)z_q(k-j)\equiv 0\pmod{\Phi_n(q)}.
\end{equation*}
Since the polynomials $(1-aq^n)(a-q^n)$ and $\Phi_n(q)$ are relatively prime, we get
\begin{align*}
\sum_{k=0}^{n-1}\sum_{j=0}^{k}x_q(j)x_q(k-j)
\equiv [n]^2\frac{(q^3;q^4)_{(n-1)/2}^2}{(q^5;q^4)_{(n-1)/2}^2} \pmod{\Phi_n(q)(1-aq^n)(a-q^n)}.
\end{align*}
Letting $a\to1$ in the above congruence, we arrive at the second case of \eqref{eq:thm6-1}.
\end{proof}

Consider the case where $n=p$ is a prime in Theorem \ref{thm6}.
By \cite[Proposition 1.3]{Wei}, we immediately obtain the following conclusion: for any odd prime $p$,
\begin{align}
\sum_{k=0}^{p-1}\frac1{64^{k}}\sum_{j=0}^{k} {2j\choose j}^3{2k-2j\choose k-j}^3
\equiv
\begin{cases}
 \Gamma_p(1/4)^8 \pmod{p^3} &\text{if $p\equiv 1\pmod{4}$},\\[5pt]
 0\pmod{p^3} &\text{if $p\equiv 3\pmod{4}$},
\end{cases} \label{eq:h-cases}
\end{align}
where $\Gamma_p(x)$ is the $p$-adic Gamma function.

It is proved in \cite{Guo-a2} and \cite{WY2} that
\begin{align}
&\sum_{k=0}^{(n-1)/2}(-1)^k[4k+1]\frac{(q;q^2)_k^4(q^2;q^{4})_k}{(q^2;q^2)_k^4(q^4;q^4)_k}q^k  \notag\\[5pt]
&\quad\equiv \begin{cases}
[n]\dfrac{(q^2;q^4)_{(n-1)/4}^2}{(q^4;q^4)_{(n-1)/4}^2} \pmod{[n]\Phi_n(q)^2} &\text{if $n\equiv1\pmod{4}$}, \\[12pt]
0\pmod{[n]\Phi_n(q)^2} &\text{if $n\equiv3\pmod{4}$},
\end{cases}
\end{align}
which is a $q$-analogue of the (A.2) supercongruence of Van Hamme \cite{Hamme}. We have the following
related $q$-supercongruences on double sums.

\begin{theorem}\label{thm7}
Let $n$ be a positive odd integer, and for $k\geqslant 0$,
$$
c_q(k)=(-1)^k[4k+1]\frac{(q;q^2)_k^4(q^2;q^{4})_k}{(q^2;q^2)_k^4(q^4;q^4)_k}q^k.
$$
Then, modulo $[n]\Phi_n(q)^2$,
\begin{align}
\sum_{k=0}^{n-1}\sum_{j=0}^{k}c_q(j)c_q(k-j)
\equiv \begin{cases}
[n]^2\dfrac{(q^2;q^4)_{(n-1)/4}^4}{(q^4;q^4)_{(n-1)/4}^4}  &\text{if}\; n\equiv1\pmod4, \\[10pt]
0&\text{if}\; n\equiv3\pmod4.
\end{cases}\label{eq:thm7-1}
\end{align}
\end{theorem}

\begin{proof}For $k\geqslant 0$, let
$$
z_q(k)=(-1)^k[4k+1]\frac{(aq,q/a,q,q;q^2)_k(q^2;q^{4})_k}{(aq^2,q^2/a,q^2,q^2;q^2)_k(q^4;q^4)_k}q^k.
$$
Using the following congruence in \cite{Guo-a2} and \cite{WY2}: modulo $[n](1-aq^n)(a-q^n)$,
\begin{align*}
\sum_{k=0}^{(n-1)/2}z_q(k)
\equiv \begin{cases}
[n]\dfrac{(q^2;q^4)_{(n-1)/4}^2}{(q^4;q^4)_{(n-1)/4}^2}  &\text{if $n\equiv1\pmod{4}$}, \\[12pt]
0 &\text{if $n\equiv3\pmod{4}$},
\end{cases}
\end{align*}
Similarly to the proof of Theorem \ref{thm1}, we can prove that, modulo $[n](1-aq^n)(a-q^n)$,
\begin{align*}
\sum_{k=0}^{n-1}\sum_{j=0}^{k}z_q(j)z_q(k-j)
\equiv \begin{cases}
[n]^2\dfrac{(q^2;q^4)_{(n-1)/4}^4}{(q^4;q^4)_{(n-1)/4}^4}  &\text{if $n\equiv1\pmod{4}$}, \\[12pt]
0 &\text{if $n\equiv3\pmod{4}$}.
\end{cases}
\end{align*}
Finally, letting $a=1$ in the above congruence, we are led to the desired $q$-supercongruence \eqref{eq:thm7-1}.
\end{proof}

The first author and Schlosser \cite[Theorem 2.1]{GS} gave the following $q$-supercongruence:
\begin{align*}
\sum_{k=0}^{(n-1)/2}[4k+1]\frac{(q;q^2)_k^6}{(q^2;q^2)_k^6} q^k
\equiv q^{(1-n)/2}[n] \sum_{k=0}^{(n-1)/2} \frac{(q;q^2)_k^4}{(q^2;q^2)_k^4} q^{2k}
\pmod{[n]\Phi_n(q)^2},
\end{align*}
which is a partial $q$-analogue of a superconruence of Long \cite{Long}. Using a parametric generalization of
the above $q$-supercongruence \cite[Theorem 3.3 with $b=1$]{GS}, we can prove the following theorem.
Since the proof is exactly the same as that of Theorem \ref{thm1}, we omit the details here.

\begin{theorem}\label{thm8}
Let $n$ be a positive odd integer, and for $k\geqslant 0$,
$$
c_q(k)=[4k+1]\frac{(q;q^2)_k^6}{(q^2;q^2)_k^6} q^k.
$$
Then
\begin{align*}
\sum_{k=0}^{n-1}\sum_{j=0}^{k}c_q(j)c_q(k-j)
\equiv q^{1-n}[n]^2 \left(\sum_{k=0}^{(n-1)/2} \frac{(q;q^2)_k^4}{(q^2;q^2)_k^4} q^{2k}\right)^2 \pmod{[n]\Phi_n(q)^2}.
\end{align*}
\end{theorem}

\section{Open problems and concluding remarks}

Letting $n=p^r$ be an odd prime power and $q\to 1$ in \eqref{eq:thm1} and \eqref{eq:thm2}, we obtain the following supercongruences:
for any odd prime $p$ and positive integer $r$,
\begin{align*}
\sum_{k=0}^{p^r-1}\frac1{16^{k}}\sum_{j=0}^{k} {2j\choose j}^3{2k-2j\choose k-j}^3(3j+1)(3k-3j+1)
&\equiv p^{2r}\pmod{p^{r+2}},\\
\sum_{k=0}^{p^r-1}\frac1{(-8)^{k}}\sum_{j=0}^{k} {2j\choose j}^3{2k-2j\choose k-j}^3(3j+1)(3k-3j+1)
&\equiv p^{2r}\pmod{p^{r+2}}.
\end{align*}

We have a conjecture related to the above two supercongruences.

\begin{conjecture} Let $p$ be an odd prime and $r$ a positive integer. Then
\begin{align*}
\sum_{k=0}^{p^r-1}\left(\frac{1}{16^k}-\frac{1}{(-8)^k}\right)\sum_{j=0}^k {2j\choose j}^3{2k-2j\choose k-j}^3 (3j+1)(3k-3j+1)
\equiv 0 \pmod{p^{2r+2}}.
\end{align*}
\end{conjecture}

Numerical evaluation indicates that Theorem \ref{thm3} can be further strengthened as follows.

\begin{conjecture}The $q$-congruence \eqref{eq:thm-a3} holds modulo $[n]^2(1-aq^n)(a-q^n)$. In particular,
the $q$-congruence \eqref{eq:thm3} holds modulo $[n]^2\Phi_n(q)^2$.
\end{conjecture}

We believe that the following generalization of Theorem \ref{thm4} should be true.

\begin{conjecture}
Let $c_q(k)$ be given in Theorem \ref{thm4}. Then
\begin{equation*}
\sum_{k=0}^{n-1}\sum_{j=0}^{k}c_q(j)c_q(k-j)
\equiv q^{1-n}[n]^2+\frac{(n^2-1)(1-q)^2}{12}q[n]^4 \pmod{[n]^2\Phi_n(q)^3}.
\end{equation*}
\end{conjecture}
Note that $q$-congruences modulo the fifth power of a cyclotomic polynomial is rare and is in general rather
difficult to prove. Another such an unsolved $q$-congruence can be found in \cite[Conjecture 6.4]{Guo-mod4}.

We have two conjectural generalizations of Theorem \ref{thm5}.
\begin{conjecture}The $q$-congruence \eqref{eq:thm5-2} holds modulo $\Phi_n(q)(1-aq^n)(a-q^n)$. In particular,
the $q$-congruence \eqref{eq:thm5-1} holds modulo $\Phi_n(q)^3$.
\end{conjecture}

\begin{conjecture}
Let $d$, $n$ and $r$ be positive integers with $\gcd(d,n)=1$ and $n$ odd.  For $k\geqslant 0$, let
$$
c_q(k)=\frac{2(q^r;q^d)_k (q^{d-r};q^d)_k q^{2dk}}{(q^d;q^d)_k^2(1+q^{dk})}.
$$
Then
\begin{equation}
\sum_{k=0}^{n-1}\sum_{j=0}^{k}c_q(j)c_q(k-j)\equiv 1\pmod{\Phi_n(q)^2}. \label{eq:conj6-5}
\end{equation}
\end{conjecture}

Note that the first author \cite[Corollary 1.4]{Guo-para} proved the following congruence:
modulo $(1-aq^{r+d\langle-r/d\rangle_n})(a-q^{d-r+d\langle (r-d)/d\rangle_n})$,
\begin{equation}
\sum_{k=0}^{n-1}\frac{2(aq^r;q^d)_k (q^{d-r}/a;q^d)_k
q^{dk}}{(q^d;q^d)_{k}(q^d;q^d)_{k} (1+q^{dk})} \equiv (-1)^{\langle-r/d\rangle_n},  \label{eq:conj6-6}
\end{equation}
where  $\langle x\rangle_m$ denotes the {\it least non-negative residue} of $x$ modulo $m$. It seems that the corresponding parametric generalization of
\eqref{eq:conj6-5} is true modulo $\Phi_n(q)(1-aq^{r+d\langle-r/d\rangle_n})$ if $\langle-r/d\rangle_n\leqslant (n-1)/2$ or
is true modulo $\Phi_n(q)(a-q^{d-r+d\langle (r-d)/d\rangle_n})$ if $\langle(r-d)/d\rangle_n\leqslant (n-1)/2$.
Using \eqref{eq:conj6-6}, we can only prove the modulus $(1-aq^{r+d\langle-r/d\rangle_n})$ or $(a-q^{d-r+d\langle (r-d)/d\rangle_n})$ case.

We have the following generalization of \eqref{eq:h-cases} for the second case.
\begin{conjecture} Let $p\equiv 3\pmod{4}$ be a prime greater than $3$ and $r$ a positive integer. Then
$$
\sum_{k=0}^{p^r-1}\frac1{64^{k}}\sum_{j=0}^{k} {2j\choose j}^3{2k-2j\choose k-j}^3
\equiv 0\pmod{p^4}.
$$
\end{conjecture}

Finally, the following stronger versions of  Theorems \ref{thm7} and \ref{thm8} appear to be true.
\begin{conjecture}
Let $c_q(k)$ be given in Theorem \ref{thm7}. Then
\begin{align*}
\sum_{k=0}^{n-1}\sum_{j=0}^{k}c_q(j)c_q(k-j)
\equiv \begin{cases}
[n]^2\dfrac{(q^2;q^4)_{(n-1)/4}^4}{(q^4;q^4)_{(n-1)/4}^4} \pmod{[n]^2\Phi_n(q)^2}  &\text{if}\; n\equiv1\pmod4, \\[12pt]
0\pmod{[n]^2\Phi_n(q)^4}  &\text{if}\; n\equiv3\pmod4.
\end{cases}
\end{align*}
\end{conjecture}

\begin{conjecture}
Let $c_q(k)$ be given in Theorem \ref{thm8}. Then, modulo $[n]^2\Phi_n(q)^3$,
\begin{align*}
\sum_{k=0}^{n-1}\sum_{j=0}^{k}c_q(j)c_q(k-j)
\equiv q^{1-n}\left([n]^2+\frac{(n^2-1)(1-q)^2}{12}[n]^4\right) \left(\sum_{k=0}^{(n-1)/2} \frac{(q;q^2)_k^4}{(q^2;q^2)_k^4} q^{2k}\right)^2.
\end{align*}
\end{conjecture}

\vskip 5mm \noindent{\bf Acknowledgment.} The first author was partially supported by the National Natural Science Foundation of
China (grant 11771175). The second author was partially supported by the Natural Science Foundation of the Jiangsu Higher
Education Institutions of China (grant 19KJB110006).

\end{document}